\documentclass[final,3p,times]{article}

\usepackage[a4paper,left=25mm,top=25mm,right=20mm,bottom=25mm]{geometry}

\usepackage{authblk}
\usepackage[numbers, sort&compress]{natbib}

\usepackage[colorlinks]{hyperref}
\hypersetup{
	citecolor=blue,
   linkcolor=red,
}

\usepackage{amsmath, amssymb, amsfonts, amsfonts, amsthm,latexsym}
\usepackage[noend]{algorithmic}
\usepackage{algorithm}

\usepackage{graphics}
\usepackage{enumerate}
\usepackage[usenames]{color}
\usepackage{mathtools}
\usepackage[normalem]{ulem}
\newtheorem{theorem}{Theorem}[section]
\newtheorem{lemma}[theorem]{Lemma}
\newtheorem{definition}{Definition}[section]
\newtheorem{corollary}[theorem]{Corollary}

\usepackage{url}
\RequirePackage{hyperref}

\makeatletter
\def\moverlay{\mathpalette\mov@rlay}
\def\mov@rlay#1#2{\leavevmode\vtop{%
    \baselineskip\z@skip \lineskiplimit-\maxdimen
    \ialign{\hfil$\m@th#1##$\hfil\cr#2\crcr}}}
\newcommand{\charfusion}[3][\mathord]{
  #1{\ifx#1\mathop\vphantom{#2}\fi
    \mathpalette\mov@rlay{#2\cr#3}
  }
  \ifx#1\mathop\expandafter\displaylimits\fi}
\DeclareRobustCommand\bigop[1]{%
  \mathop{\vphantom{\sum}\mathpalette\bigop@{#1}}\slimits@
}
\newcommand{\bigop@}[2]{%
  \vcenter{%
    \sbox\z@{$#1\sum$}%
    \hbox{\resizebox{\ifx#1\displaystyle.9\fi\dimexpr\ht\z@+\dp\z@}{!}{$\m@th#2$}}%
  }%
}
\makeatother
\newcommand{\bigjoin}{\bigop{\triangledown}}
\DeclareMathOperator{\join}{\triangledown}
\newcommand{\cupdot}{\charfusion[\mathbin]{\cup}{\cdot}}
\DeclareMathOperator{\bigcupdot}{\charfusion[\mathop]{\bigcup}{\cdot}}
\definecolor{jade}{rgb}{0.0, 0.66, 0.42}
\newcommand{\child}{\mathsf{child}}
\DeclareMathOperator*{\argmax}{arg\,max}

\definecolor{darkorchid}{rgb}{0.6, 0.2, 0.8}

\newcommand{\AX}[1]{\textnormal{#1}}

\providecommand{\keywords}[1]{\textbf{\textit{Keywords: }} #1}

\begin{document}

\title{Hierarchical Colorings of Cographs}

\author[AGU]{Dulce I. Valdivia}

\author[LEI]{Manuela Gei{\ss}}

\author[GRE,SAR]{Marc Hellmuth} 

\author[JUR]{Maribel Hern{\'a}ndez Rosales}

\author[LEI,LEI-other,MIS,TBI,BOG,SFI]{Peter F. Stadler}    

\affil[AGU]{Universidad Aut{\'o}noma de Aguascalientes,
  Centro de Ciencias B{\'a}sicas,
  Av.\ Universidad 940, 20131 Aguascalientes, AGS, M{\'e}xico}

\affil[LEI]{Bioinformatics Group, Department of Computer Science \&
  Interdisciplinary Center for Bioinformatics, Universit{\"a}t Leipzig,
  H{\"a}rtelstra{\ss}e 16-18, D-04107 Leipzig, Germany}

\affil[GRE]{Institute of Mathematics and Computer Science, University of
  Greifswald, Walther-Rathenau-Stra{\ss}e 47, D-17487 Greifswald, Germany}

\affil[SAR]{Center for Bioinformatics, Saarland University, Building E 2.1, 
  P.O.\ Box 151150, D-66041 Saarbr{\"u}cken, Germany}

\affil[JUR]{CONACYT-Instituto de Matem{\'a}ticas, UNAM Juriquilla, Blvd.\
  Juriquilla 3001, 76230 Juriquilla, Quer{\'e}taro, QRO, M{\'e}xico}

\affil[LEI-other]{German Centre for Integrative Biodiversity Research
  (iDiv) Halle-Jena-Leipzig, Competence Center for Scalable Data Services
  and Solutions Dresden-Leipzig, Leipzig Research Center for Civilization
  Diseases, and Centre for Biotechnology and Biomedicine at Leipzig
  University at Universit{\"a}t Leipzig}

\affil[MIS]{Max Planck Institute for Mathematics in the Sciences,
  Inselstra{\ss}e 22, D-04103 Leipzig, Germany} 

\affil[TBI]{Institute for Theoretical Chemistry, University of Vienna,
  W{\"a}hringerstrasse 17, A-1090 Wien, Austria}

\affil[BOG]{Facultad de Ciencias, Universidad National de Colombia, Sede
  Bogot{\'a}, Colombia}

\affil[SFI]{Santa Fe Insitute, 1399 Hyde Park Rd., Santa Fe NM 87501,
  USA}

\date{\ }

\maketitle

\abstract{ 
  Cographs are exactly hereditarily well-colored graphs, i.e., the graphs
  for which a greedy coloring of every induced subgraph uses only the
  minimally necessary number of colors $\chi(G)$. In recent work on
  reciprocal best match graphs so-called hierarchically coloring play an
  important role. Here we show that greedy colorings are a special case of
  hierarchical coloring, which also require no more than $\chi(G)$ colors.
}

%

\section{Introduction and Preliminaries}

Let $G=(V,E)$ be an undirected graph. A (proper vertex) coloring of $G$ is
a surjective function $\sigma: V\to S$ such that $xy\in E$ implies
$\sigma(x)\ne \sigma(y)$. The minimum number $|S|$ of colors such that
there is a coloring of $G$ is known as the \emph{chromatic number}
$\chi(G)$.  A \emph{greedy coloring} of $G$ is obtained by ordering the set
of colors and coloring the vertices of $G$ in a random order with the first
available color. The \emph{Grundy number} $\gamma(G)$ is the maximum number
of colors required in a greedy coloring of $G$
\cite{Christen:79}. Obviously $\gamma(G)\ge\chi(G)$. Determining $\chi(G)$
\cite{Karp:72} and $\gamma(G)$ \cite{Zaker:06} are NP-complete problems.  A
graph $G$ is called \emph{well-colored} if $\chi(G)=\gamma(G)$
\cite{Zaker:06}. It is \emph{hereditarily well-colored} if every induced
subgraph is well-colored.

\begin{definition}[\cite{Corneil:81}]
  \label{def:Corneil}
  A graph $G$ is a \emph{cograph} if $G=K_1$, $G$ is the disjoint union
  $G=\bigcupdot_i G_i$ of cographs $G_i$, or $G$ is a join
  $G=\bigjoin_i G_i$ of cographs $G_i$.
\end{definition}
This recursive construction induces a rooted tree $T$, whose leaves are
individual vertices corresponding to a $K_1$ and whose interior vertices
correspond to the union and join operations. We write $L(T)$ for the leaf
set and $V^0(T)$ for the set of inner vertices of $T$. The set of children
of $u$ is denoted by $\child(u)$. For edges $e=uv$ in $T$ we adopt the
convention that $v$ is a child of $u$. We define a labeling function
$t:V^0(T)\rightarrow\{0,1\}$, where an interior vertex $u$ of $T$ is
labeled $t(u)=0$ if it is associated with a disjoint union, and $t(u)=1$
for joins. The set $L(T(u))$ denotes the leaves of $T$ that are descendants
of $u$.  To simplify the notation we will write $G(u):=G[L(T(u))]$ for the
subgraph of $G$ induced by the vertices in $L(T(u))$. Note that $G(u)$ is
the graph consisting of the single vertex $u$ if $u$ is a leaf of $T$.

Given a cograph $G$, there is a unique \emph{discriminating}
cotree\footnote{In \cite{Corneil:81} the discriminating cotree is defined
  as \emph{the} cotree associated with $G$. Here we call every tree arising
  from Def.~\ref{def:Corneil} \emph{a} cotree of $G$.} in which adjacent
operations are distinct, i.e., $t(u)\neq t(v)$ for all interior edges
$uv\in E(T)$.  It is possible to refine the discriminating cotree by
subdiving a disjoint union or join into multiple disjoint unions or joins,
respectively \cite{Boecker:98,Corneil:81}. It is well known that every
induced subgraph of a cograph is again a cograph. A graph is a cograph if
and only if it does not contain a path $P_4$ on four vertices as an induced
subgraph \cite{Corneil:81}.  The cographs are also exactly the hereditarily
well-colored graphs \cite{Christen:79}. The chromatic number of a cograph
$G$ can be computed recursively, as observed in
\cite[Tab.1]{Corneil:81}. Starting from $\chi(K_1)=1$ as base case we have
\begin{equation}
  \begin{split}
    \chi(G) &= \chi\left(\bigcupdot_{i} G_i\right) = \max_{i} \chi(G_i)
    \textrm{ or }\\
    \chi(G) &= \chi\left(\bigjoin_{i} G_i\right) = \sum_{i} \chi(G_i)
  \end{split}
  \label{eq:rec1}
\end{equation}

\emph{Hierarchically colored cograph} (\emph{hc-cographs}) were introduced
in \cite{Geiss:19x} as the undirected colored graphs recursively defined by
 \begin{description}
 \item[\AX{(K1)}] $(G,\sigma)=(K_1,\sigma)$, i.e., a colored vertex, or
 \item[\AX{(K2)}] $(G,\sigma)= (H_1,\sigma_{H_1}) \join (H_2,\sigma_{H_2})$
   and $\sigma(V(H_1))\cap \sigma(V(H_2))=\emptyset$, or
 \item[\AX{(K3)}]
   $(G,\sigma)= (H_1,\sigma_{H_1}) \cupdot (H_2,\sigma_{H_2})$ and 
   $\sigma(V(H_1))\cap \sigma(V(H_2)) \in
   \{\sigma(V(H_1)),\sigma(V(H_2))\}$,
\end{description}
where $\sigma(x) = \sigma_{H_i}(x)$ for every $x\in V(H_i)$, $i\in\{1,2\}$
and $(H_1,\sigma_{H_1})$ and $(H_2,\sigma_{H_2})$ are hc-cographs.

This recursive construction of an hc-cograph $G$ implies a binary \emph{cotree} $(T,t)$. Its inner vertices can be associated with the
intermediate graphs in the construction. We say that $\sigma$ is an
\emph{hc-coloring w.r.t.\ $(T,t)$.}

Obviously, the graph $G$ underlying an hc-cograph is a cograph.
\begin{definition}\label{def:hc-coloring}
  Let $G=(V,E)$ be a cograph. A coloring $\sigma:V\to S$ of $G$ is an
  \emph{hc-coloring} of $G$ if there is binary cotree $(T,t)$ of $G$ such
  that $(G,\sigma)$ is hc-cograph w.r.t.\ $(T,t)$.
\end{definition}
This contribution aims to investigate the properties of hc-colorings and
their relationships with other types of cograph colorings.

\section{Existence of hc-Colorings}

As noticed in \cite{Geiss:19x}, a coloring $\sigma$ of a cograph $G$ may be
an hc-coloring w.r.t.\ some cotree $(T,t)$ but not w.r.t.\ to another
cotree $(T',t')$ that yields the same cograph. An example is shown in
Fig. \ref{Fig:diffcotree}.

\begin{figure}[h]
  \begin{center}
    \includegraphics[scale=0.55]{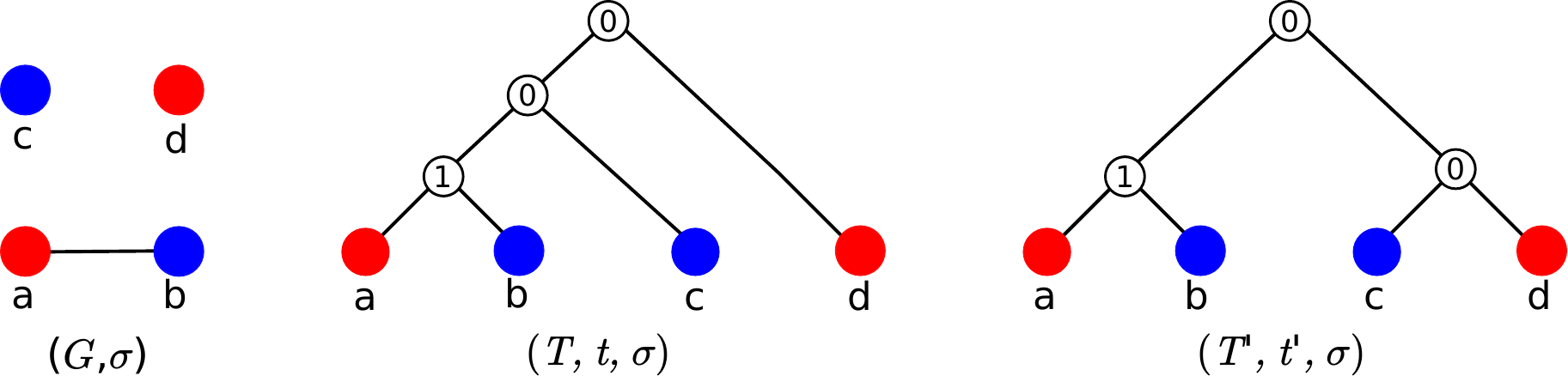}		
  \end{center}
  \caption{The induced cotree of a cograph $G$ affects the hc-coloring
    property of $\sigma$, where
    $\sigma(a)=\sigma(d)\neq\sigma(b)=\sigma(c)$. In $(T,t)$, the first
    tree from left to right, \AX{(K1)}-\AX{(K3)} are satisfied making
    $\sigma$ an hc-coloring.  However the second tree $(T',t')$ does not
    satisfy \AX{(K3)} since the parent node of $c\simeq K_1$ and
    $d\simeq K_1$ corresponds to a disjoint union operation and
    $\sigma(c)\cap\sigma(d)=\emptyset$. Thus $\sigma$ is not an hc-coloring
    w.r.t. $(T',t')$.}
  \label{Fig:diffcotree}
\end{figure}

\begin{theorem}
  Let $\sigma:V\to S$ be an hc-coloring of a cograph $G$. Then
  $|S|=\chi(G)$.
  \label{thm:hc-implies-chrom}
\end{theorem}
\begin{proof}
  We proceed by induction w.r.t.\ $|V|$. The statement is trivially true
  for $|V|=1$, i.e. $G=K_1$, since $\chi(K_1)=1$. Now suppose $|V|>1$. Thus
  $G=G_1\join G_2$ or $G=G_1 \cupdot G_2$ for some graphs $G_1=(V_1,E_1)$
  and $G_1=(V_2,E_2)$ with $1\leq|V_1|,|V_2|<|V|$. By induction
  hypothesis we have $|\sigma(V_1)|=\chi(G_1)$ and
  $|\sigma(V_2)|=\chi(G_2)$.

  First consider $G=G_1\join G_2$. Since $xy\in E(G)$ for all
  $x\in V_1$ and $y\in V_2$ we have $\sigma(x)\ne\sigma(y)$, and hence
  $\sigma(V_1)\cap\sigma(V_2)=\emptyset$. Thus,
  $\sigma(V)=\sigma(V_1)\cupdot\sigma(V_2)$ and therefore,
  \begin{equation*}
    |\sigma(V)| = |\sigma(V_1)| + |\sigma(V_2)| \stackrel{\text{IH.}}{=}
    \chi(G_1)+\chi(G_2) \stackrel{\text{Equ.\ }\eqref{eq:rec1}}{=} \chi(G).
  \end{equation*}
  We note that the coloring condition in \AX{(K2)} therefore only enforces
  that $\sigma$ is a proper vertex coloring.

  Now suppose $G=G_1\cupdot G_2$. Axiom \AX{(K3)} implies
  $|\sigma(V)|=\max( |\sigma(V_1)|,|\sigma(V_2)|)$. Hence,
  \begin{equation*}          
    |\sigma(V)| = \max( |\sigma(V_1)|,|\sigma(V_2)|)
    \stackrel{\text{IH.}}{=}
    \max(\chi(G_1),\chi(G_2)) \stackrel{\text{Equ.\ }\eqref{eq:rec1}}{=}
    \chi(G).
  \end{equation*}
\end{proof}

As detailed in \cite{Christen:79}, we have $\chi(G)=\gamma(G)$. Thus, it
seems natural to ask whether every greedy coloring is an
hc-coloring. Making use of the fact that  $\chi(G)=\gamma(G)$, we assume
w.l.o.g.\ that the color set is $S=\{1,2,\dots,\chi(G)\}$ whenever we
consider greedy colorings of a cograph. By definition of greedy colorings
and the fact that cographs are hereditarily well-colored, we immediately
observe
\begin{lemma}
  Let $\sigma$ be a greedy coloring of a cograph $G$ and
  $G_1=(V_1,E_1)$ a connected component of $G$. Then $G_1$ is colored by
  $\sigma(V_1)=\{1,\dots,\chi(G_1)\}$.
\end{lemma}

We shall say that a cograph $G$ is a \emph{minimal counterexample for some
property $\mathcal{P}$} if (1) $G$ does not satisfy $\mathcal{P}$ and (2)
every induced subgraph of $G$ (i.e., every ``smaller'' cograph) satisfies
$\mathcal{P}$. 

\begin{lemma}
  \label{lem:grehc}
  Let $G$ be a cograph, $(T,t)$ an arbitrary binary cotree for $G$
  and $\sigma$ a greedy coloring of $G$. Then $\sigma$ is an
  hc-coloring w.r.t.\ $(T,t)$.
\end{lemma}
\begin{proof}
  Assume $G$ is a minimal counterexample, i.e., $G$ is a minimal cograph
  for which a coloring $\sigma$ exists that is a greedy coloring but not an
  hc-coloring. If $G$ is connected, then $G=\bigjoin_{i=1}^n G_i$, for some
  $n>1$ and $\sigma(V)=\bigcupdot_{i=1}^n \sigma(V_i)$, i.e., \AX{(K2)} is
  satisfied.  By assumption, $\sigma$ is not an hc-coloring, hence $\sigma$
  must fail to be an hc-coloring on at least one of the connected
  components $G_i$, contradicting the assumption that $G$ is a minimal
  counterexample. Thus, $G$ cannot be connected.

  Therefore, assume $G=\bigcupdot_{i=1}^n G_i$ for some $n>1$. Since $G$ is
  represented by a binary cotree $(T,t)$, the root of $T$ must have exactly
  two children $u$ and $v$. Hence, we can write $G=G(u)\cupdot G(v)$.
  Since $G$ is a minimal counterexample, we can conclude that $\sigma$
  induces an hc-coloring on $G(u)$ and $G(v)$.  However, since $\sigma$ is,
  in particular, a greedy coloring of $G(u)$ and $G(v)$,
  $\sigma(V(G(u)))\subseteq \sigma(V(G(v)))$ or
  $\sigma(V(G(v)))\subseteq \sigma(V(G(u)))$ most hold. But this
  immediately implies that $(G,\sigma)$ satisfies $\AX{(K3)}$ and thus
  $\sigma$ is an hc-coloring of $G$.  Therefore, $G$ is not a minimal
  counterexample, which completes the proof.
\end{proof}

As an immediate consequence we find
\begin{corollary}
  Every cograph has an hc-coloring.
\end{corollary}
The converse of Lemma \ref{lem:grehc} is not true.  Fig.\
\ref{Fig:hcnogrundy} shows an example of an hc-coloring that is not a greedy
coloring.

\begin{figure} 
  \begin{tabular}{ccc}
    \begin{minipage}{0.3\textwidth}
      \includegraphics[width=\textwidth]{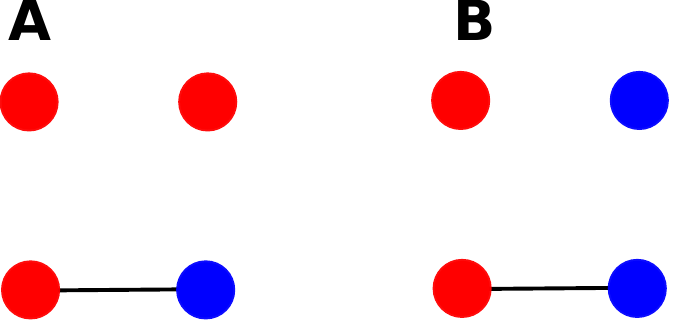}
    \end{minipage} &&
    \begin{minipage}{0.6\textwidth}
      \caption{Both colorings \textsf{A} and \textsf{B} of
        $K_2\cupdot K_1\cupdot K_1$ are hc-colorings. However, only
        \textsf{A} is a greedy coloring since the two single-vertex
        components have different colors in \textsf{B}.}
    \end{minipage} 
  \end{tabular}
  \label{Fig:hcnogrundy}
\end{figure}

\begin{theorem}
  A coloring $\sigma$ of a cograph $G$ is greedy coloring if and only if it
  is an hc-coloring w.r.t.\ every binary cotree $(T,t)$ of $G$.
\end{theorem}
\begin{proof}
  By Lemma~\ref{lem:grehc}, every greedy coloring is an hc-coloring for
  every binary cotree $(T,t)$. Now suppose $\sigma$ is an hc-coloring for
  every binary cotree $(T,t)$ and let $G$ be a minimal cograph for which
  $\sigma$ is not a greedy coloring. As in the proof of Lemma
  \ref{lem:grehc} we can argue that $G$ cannot be a minimal counterexample
  if $G$ is connected: in this case, $G=\bigjoin_{i=1}^n G_i$ and
  $\sigma(V)=\bigcupdot_{i=1}^n \sigma(V_i)$ for all colorings, and thus
  $\sigma$ is a greedy coloring if and only if it is a greedy coloring with
  disjoint color sets for each $G_i$. Hence, a minimal counterexample must
  have at least two connected components.

  Let $G=\bigcupdot_{i=1}^n G_i$ for some $n>1$ and define a partition of
  $\{1,\dots,n\}$ into sets $I_1,\dots,I_m$, $m\geq 1$, such that for every
  $r\in \{1,\dots,m\}$, $i,j\in I_r$ if and only if
  $\chi(G_i) = \chi(G_j)$. Since every
  $G^r \coloneqq \bigcupdot_{i\in I_r} G_i$, $1\leq r\leq m$, is a cograph,
  each $G^r$ can be represented by a (not necessarily unique) binary cotree
  $(T^r,t^r)$.  Note, we have $\chi(G^r)=\chi(G_i)$ for all $i\in
  I_r$. Now, we can construct a binary cotree $(T,t)$ for $G$ as follows:
  let $T^*$ be a caterpillar with leaf set $L(T^*)=\{l_1,\dots,l_\ell\}$.
  We choose $T^* = (\ldots((l_1,l_2),l_3), \ldots l_\ell)$ (in Newick
  format). Note, if $\ell=1$, then $T^*\simeq K_1$.  Now, the root of every
  tree $T^1,\dots,T^r$ is identified with a unique leaf in $L(T^*)$ such
  that the root of $T_i$ is identified with $l_i\in L(T^*)$ and the
  root of $T_j$ is identified with $l_j\in L(T^*)$, where $i<j$ if and
  only if $\chi(G^i)<\chi(G^j)$. This yields the tree $T$.  The labeling
  $t$ for $(T,t)$ is provided by keeping the labels of each $(T^r,t^r)$ and
  by labeling all other inner vertices of $T$ by $0$.  It is easy to see
  that $(T,t)$ is a binary cotree for $G$. By hypothesis, this in
  particular implies that $\sigma$ is an hc-coloring w.r.t.\ $(T,t)$. We
  denote by $C(T^*)\subseteq V(T)$ the set of inner vertices of
  $T^*$. Since $\sigma$ is an hc-coloring w.r.t.\ $(T,t)$ and thus in
  particular w.r.t.\ any subtree $(T^r,t^r)$, we have
  $\sigma(V(G_i))\cap \sigma(V(G_j))\in \{\sigma(V(G_i)),\sigma(V(G_j))\}$
  for any $i,j\in I_r$, $1\le r\le m$, by \AX{(K3)}. Hence, as
  $\chi(G^r)=\chi(G_i)=\chi(G_j)$, it must necessarily hold
  $\sigma(V(G_i))= \sigma(V(G_j))$ for all $i,j\in I_r$, i.e., all
  connected components $G_i$ with the same chromatic number are colored by
  the same color set. By construction, at every node $v\in C(T^*)$ of the
  caterpillar structure, with children $v'$ and $v''$, the components
  $G':=G(v')$ and $G'':=G(v'')$ satisfy $\chi(G')<\chi(G'')$.  Invoking
  \AX{(K3)} we therefore have $\sigma(G')\subset\sigma(G'')$ and
  $G^*:=G(v)$ is colored by the color set $\sigma(G^*)=\sigma(G'')$.  These
  set inclusions therefore imply a linear ordering of the colors such that
  colors in $\sigma(G')$ come before those in
  $\sigma(G'')\setminus \sigma(G')$. Thus $\sigma$ is a greedy coloring
  provided that the restriction of $\sigma$ to each of the connected
  components $G_i$ of $G$ is a greedy coloring, which is true due to the
  assumption that $G$ is a minimal counterexample. Thus no minimal
  counterexample exists, and the coloring $\sigma$ is indeed a greedy
  coloring of $G$.
\end{proof}

Given an hc-cograph, it is not difficult to recover a corresponding binary
cotree. To this end, we proceed top down. Denote the root of $(T,t)$ by
$r$. It is associated with the graph $G(r)=G$. In the general step we
consider an induced subgraph $G(u)$ of $G$ associated with a vertex $u$ of
$T$. If $G(u)$ is connected, then $t(u)=1$ and $G(u)$ is the joint of pair
of induced subgraphs $G(v_1)$ and $G(v_2)$. To identify these graphs,
consider the connected components $\overline{G_1},\dots,\overline{G_k}$ of
the complement $\overline{G(u)}$ of $G(u)$. We have
\begin{equation}
  G(u) = \overline{\bigcupdot_{i=1}^k \overline{G_i}} =
  \bigjoin_{i=1}^k \overline{\overline{G_i}} = 
  \bigjoin_{i=1}^k G_i \,.
\end{equation}
We therefore set $G(v_1)=G_1$ and
$G(v_2)=\overline{\bigcupdot_{i=2}^k \overline{G_i}}=\bigjoin_{i=2}^k
G_i$. By construction, we therefore have $G(u)=G(v_1)\join G(v_2)$ with
disjoint color sets $\sigma(V(G(v_1)))$. If $G(u)$ is disconnected, define
$t(u)=0$, identify one of the components, say $G_1$, with the smallest
numbers of colors $|\sigma(V(G_1))|$ and set $G(v_1)=G_1$ and
$G(v_2)=G(u)\setminus G(v_1)$. The fact that $G(u)$ is an hc-cograph
ensures that $\sigma(V(G(v_1)))\subseteq \sigma(V(G(v_2)))$. In both the
connected and the disconnected case we attach $v_1$ and $v_2$ as the
children of $u$ in $T$. The reconstruction of $(T,t)$ can be preformed in
linear time.

\section{Recursively-Minimal Colorings}

Not every minimal coloring of a cograph is an hc-coloring. For instance, if
$G$ is a disconnected cograph, i.e., $G=\bigcupdot_{i=1}^n G_i$, then it
suffices that $|\sigma(V_i)|\le\chi(G)$. In this case, we may use more
colors than necessary on a connected component $G_i$ of $G$, resulting in
$\chi(G_i)\neq |\sigma(V_i)|$. Thus, Theorem \ref{thm:hc-implies-chrom}
implies that $\sigma$ is not an hc-coloring of $G_i$ and hence, by
definition, not an hc-coloring of $G$.  This suggests to consider another
class of colored cographs.
\begin{definition}
  A color-minimal cograph $(G,\sigma)$ is either a $K_1$, the disjoint
  union of color-minimal cographs or the join of color-minimal cographs,
  and satisfies $|\sigma(V)|=\chi(G)$. A coloring $\sigma$ of a
  color-minimal cograph will be called \emph{recursively minimal}.
\end{definition}
Color-minimal cographs thus are those colorings for which every constituent
in their construction along \emph{some} binary cotree is colored with the
minimal number of colors. Since every greedy coloring of every cograph
satisfies this condition, every cograph has a recursively minimal coloring.

\begin{theorem}
  Let $G$ be a cograph. A coloring $\sigma$ of $G$ is recursively minimal
  if and only it is an hc-coloring.
\label{thm:mrc}
\end{theorem}
\begin{proof}
  Since every hc-coloring $\sigma$ of a cograph $G$ uses exactly $\chi(G)$
  colors, the recursive definition of hc-colorings immediately implies that
  $\sigma$ is recursively minimal.

  Now suppose there is a minimal cograph $G$ with a coloring $\sigma$ that
  is recursively minimal but not an hc-coloring. If $G$ is connected, then
  $G=\bigjoin_{i=1}^n G_i$ for some $n\ge 2$ and the restrictions of
  $\sigma$ to the connected components $G_i$ use disjoint color
  sets. Hence, $\sigma$ is an hc-coloring whenever the restriction to each
  $G_i$ is an hc-coloring.  Thus a minimal counterexample cannot be
  connected. Now suppose $G=\bigcupdot_{i=1}^n G_i$ for some $n\ge
  2$. Since $(G,\sigma)$ is by assumption a minimal counterexample, each
  connected component $(G_i,\sigma_i)$ is an hc-cograph.  By Equ.\
  \eqref{eq:rec1} there is a connected component, say w.l.o.g.\ $G_1$, such
  that $\chi(G)=\chi(G_1)$. By definition, $\sigma$ induces a recursively
  minimal coloring $\sigma_1$ on $G_1$ and $\sigma'$ on
  $G'=\bigcupdot_{i=2}^n G_i$.  Since $G$ is a minimal counterexample,
  $\sigma_1$ and $\sigma'$ are hc-colorings of $G_1$ and $G'$,
  respectively, in other words $(G_1,\sigma_1)$ and $(G',\sigma')$ are
  hc-cographs. Moreover, $\chi(G)=\chi(G_1)$ implies
  $\sigma'(V(G'))\subseteq \sigma_1(V(G_1))$. In summary, therefore,
  $(G,\sigma) = (G_1,\sigma_1) \cupdot (G',\sigma')$ satisfies \AX{(K3)},
  thus it is a cograph with hc-coloring $\sigma$.  Hence, there cannot
  exist a minimal cograph with a coloring $\sigma$ that is recursively
  minimal but not an hc-coloring.
\end{proof}

Recursively minimal colorings can be constructed in a very simply manner by
stepwisely relabeling colors of disconnected subgraphs as outlined in
Alg.~\ref{alg:cBMG}.

\begin{algorithm}
  \caption{Recursively minimal coloring of a cograph}
  \label{alg:cBMG}
  \algsetup{linenodelimiter=}
  \begin{algorithmic}[1]
    \REQUIRE Cograph $G$
    \STATE $(T_G,t_G) \leftarrow$ discriminating cotree of $G$ 	
    \STATE initialize all $v \in V(G)$ with different colors
    \FORALL {$v\in V^{0}(T_G)$, from bottom to top}
       \IF {$t(v) = 0$} 
          \STATE $\mathcal{G} \leftarrow$ set of connected components of $G(v)$
          \STATE $G^* \leftarrow \argmax_{G_i\in\mathcal{G}} \chi(G_i)$
          \STATE $S \leftarrow \sigma(V(G^*))$ 
          \FORALL {$G_j\in\mathcal{G}, G_j\neq G^*$} 
             \STATE randomly choose an injective map
                      $\phi:\sigma(V(G_j))\to S$
             \FORALL {$x\in V(G_j)$}
                \STATE $\sigma(x)\leftarrow \phi(\sigma(x))$  
             \ENDFOR
          \ENDFOR
       \ENDIF
    \ENDFOR
  \end{algorithmic}
\end{algorithm}

\begin{theorem}\label{Thm:Algo1}
  Given a cograph $G$, Algorithm \ref{alg:cBMG} returns a recursively
  minimal coloring of $G$. Moreover, every recursively minimal coloring of
  a cograph $G$ can be constructed with this algorithm.
\end{theorem}
\begin{proof}
  The bottom-up traversal of the cotree ensures that for every inner vertex
  $v$ of $T_G$, the subgraphs induced by its children ${v_1,...,v_j}$ are
  color-minimal cographs. In particular this means that
  $|\sigma(V(G(v_i))|=\chi(G(v_i))$ for all $i\in \{1,..,j\}$. Take an
  arbitrary $v\in V^0(G)$. Suppose $t(v)=0$.  The fact that $T_G$ is a
  discriminating cotree implies that all $t(v_i)=1$, for all
  $i\in \{1,..,j\}$, and therefore $G_i := G(v_i)$ are connected
  components. Futhermore, there exists a $G_k$ such that
  $\chi(G_k)=\max_{i}\chi(G_i)$. These observations and Equ.\
  \eqref{eq:rec1} guarantee that lines 5-7 obtain a set of color $S$ such
  that $|S|=\chi(G(v))$. Explicitly, $\chi(G(v))=\chi(G_k)$ and
  $S=\sigma(V(G_k))$. An injective recoloring
  $\phi:\sigma(V(G_i)) \rightarrow S$, in lines 8-11, assures that
  $\sigma(G_i)\subseteq S$ and therefore $\sigma(V(G(v))=S$. This implies
  that $G(v)$ is a color-minimal cograph with $\sigma$ a recursively
  minimal coloring.  The converse is followed by the fact that every
  recursively minimal coloring can be obtained with a particular injection
  $\phi$.
\end{proof}

Algorithm \ref{alg:cBMG} can be modified easily to construct a recursively
minimal coloring of $G$ with respect to a user defined cotree $(T,t)$.  It
suffices to replace the connected components of $G(v)$ by the (not
necessarily connected) induced subgraphs $G(w)$ corresponding to the
children of $v$. Since $\chi(G(v))=\max_{w\in\child(v)}\chi(G(w))$, it
suffices to choose the color set of the child that uses the largest number
of colors and re-color all other child-graphs with this color set. For
completeness, we summarize this variant in Algorithm \ref{alg:cBMG2}.

\begin{algorithm}
  \caption{Recursively minimal coloring of a cograph w.r.t.\ a given
    cotree.}
  \label{alg:cBMG2}
  \algsetup{linenodelimiter=}
  \begin{algorithmic}[1]
    \REQUIRE Cograph $G$ and co-tree $(T,t)$
    \STATE initialize all $v \in V(G)$ with different colors
    \FORALL {$v\in V^{0}(T)$, from bottom to top}
       \IF {$t(v) = 0$} 
          \STATE $\mathcal{G} \leftarrow \{G(w)\colon w\in\child(v)\}$ 
          \STATE $G^* \leftarrow \argmax_{w\in\child(v)} |\chi(G(w))|$
          \STATE $S \leftarrow \sigma(V(G^*))$ 
          \FORALL {$G_j\in\mathcal{G}, G_j\neq G^*$} 
             \STATE randomly choose an injective map
                      $\phi:\sigma(G_j)\to S$
             \FORALL {$x\in V(G_j)$}
                \STATE $\sigma(x)\leftarrow \phi(\sigma(x))$  
             \ENDFOR
          \ENDFOR
       \ENDIF
    \ENDFOR
  \end{algorithmic}
\end{algorithm}

The recursive structure of hc-cographs can also be used to count the number
of distinct hc-colorings of a cograph $G$ that is explained by a cotree
$(T,t)$. For an inner vertex $u$ of $T$ denote by $N(G(u))$ the number of
hc-colorings of $G(u)$. If $u$ is a leaf, then $N(u)=1$.  Recall that $T$
is binary by Def.\ \ref{def:hc-coloring}, i.e.,
$\child(u)=\{v_1,v_2\}$. For $t(u)=1$, we have $N(G(u))=N(G(v_1))N(G(v_2))$
since the color sets are disjoint. If $t(u)=0$, assume, w.l.o.g.\
$s_1:=|\sigma(G(v_1))|\le|\sigma(G(v_2))|=:s_2$,
$N(G(u))=N(G(v_1)) N(G(v_2)) g(s_1,s_2)$, where $g(s_1,s_2)$ is the number
of injections between a set of size $s_1$ into a set of size $s_2$, i.e.,
$g(s_1,s_2)=\binom{s_2}{s_1} s_1!$.

The total number of hc-colorings can be obtained by considering a
  caterpillar tree for the step-wise union of connected components. For
  each connected component $G_i$ with $\chi(G_i)=s_i$, and $s=\max_i s_i$
  there are $\binom{s}{s_i}$ choices of the colors, i.e., $g(s,s_i)$
  injections and thus $N(G)=\prod_i g(s,s_i) N(G_i)$ colorings. We
  note in passing that the chromatic polynomial of a cograph, and thus the
  number of colorings using the minimal number of colors, can be computed
  in polynomial time \cite{Makowsky:06}. There does not seem to be an obvious
  connection between the hc-colorings and the chromatic polynomial,
  however.

\section{Concluding Remarks}

The cotrees $(T,t)$ associated with a cograph are a special case of the
modular decomposition tree \cite{Gallai:67}, which in addition to disjoint
unions and joins also contains so-called prime nodes. The latter have a
special structure known as spiders, which also admit a well-defined unique
decomposition in so-called $P_4$-sparse graphs \cite{Jamison:92}. For this
type of graphs it also makes sense to consider recursively minimal
colorings. More generally, many interesting classes of graphs admit
recursive constructions \cite{Proskurowski:81,Noy:04}. For every graph
class that has a recursive construction, one can ask whether minimal
colorings can be constructed from optimal colorings, i.e., whether
recursively minimal colorings exist. In some cases, such Cartesian products
of graphs, where $\chi(G)$ equals the maximum of the chromatic numbers of
the factors \cite{Sabidussi:57}, this seems rather straightforward. In
general, however, the answer is probably negative.

\subsection*{Acknowledgments} 

This work was support in part by the German Federal Ministry of Education
and Research (BMBF, project no.\ 031A538A, de.NBI-RBC) and the Mexican
Consejo Nacional de Ciencia y Tecnolog{\'i}a (CONACyT, 278966 FONCICYT2).

\bibliographystyle{plain}
\bibliography{cographcoloring}
\end{document}